\documentclass[a4paper,11pt]{amsart}
\usepackage[plainpages=false]{hyperref}
\usepackage{amsfonts,latexsym,rawfonts,amsmath,amssymb,amsthm, mathrsfs, lscape}
\usepackage{verbatim}

\usepackage[all]{xy}
\usepackage{graphicx,psfrag}

\usepackage{array,tabularx}

\usepackage{setspace}

\newtheorem{thm}{Theorem}

\newtheorem{prop}{Proposition}[section]
\newtheorem{q}[thm]{Question}

\theoremstyle{remark}

\theoremstyle{definition}

\numberwithin{equation}{section}

\def\R{\mathbb{R}}
\def\C{\mathbb{C}}

\def \T {\mathbb T}

\def\l{\lambda}

\def\z{\mathfrak{z}}

\def \t {\mathfrak t}

\def\g{{\mathfrak g}}

\def\h{{\mathfrak h}}

\def\cF{{\mathcal F}}

\def\cH{{\mathcal H}}

\def\cR{{\mathcal R}}

\begin{document}

\title[Isometry group of  Sasaki-Einstein metric]{ Isometry group of  Sasaki-Einstein metric}

\author{Weiyong He}

\address{Department of Mathematics, University of Oregon, Eugene, Oregon, 97403}
\email{whe@uoregon.edu}

\maketitle

Let $(M, g)$ be a Sasaki-Einstein manifold of dimension $2n+1$; equivalently its Kahler cone is a Kahler-Ricci flat cone.  Let $(X, J)$  be the underlying affine variety of its  Kahler cone and denote $\text{Aut}(X, J)$ to be its automorphism group; denote $\text{Aut}_0(X, J)$ to be the identity component of $\text{Aut}(X, J)$. We prove the following result in this short paper,

\begin{thm}\label{T-main}The identity component of the holomorphic isometry group of $(M, g)$ is  the identity component of a maximal compact subgroup of $\text{Aut}(X, J)$.  
\end{thm}

This answers a conjecture proposed in Martelli-Sparks-Yau \cite{MSY} about the holomorphic isometry group of a Sasaki-Einstein metric; when a Sasaki-Einstein metric is \emph{quasiregular}, 
this is proved in Martelli-Sparks-Yau (\cite{MSY}, Section 4.3). The statement itself can be viewed as a generalization of Mastushima's theorem \cite{Matsu} on a Kahler-Einstein metric on a Fano manifold, which asserts that  the identity component of the isometry group of a Kahler-Einstein metric on a Fano manifold is the identity component of  a maximal compact subgroup of its automorphism group. Unlike Fano case, a killing vector field of a Sasaki-Einstein metric does not have to be holomorphic; hence we can only assert the conclusion about holomorphic isometry group. A typical example is the odd dimensional $(2n+1)$ round sphere whose identity component of isometry group is $SO(2n+2)$, but the holomorphic isometry group is $U(n+1)$.  By a general result on Sasaki manifolds (see Theorem 8.18, Corollary 8.19 in \cite{BG}), a Killing vector field of a Sasaki-Einstein metric is  (real) holomorphic unless on a round sphere or a 3-Sasaki structure (its Kahler cone is a hyper-Kahler con and this is the counterpart of hyper-Kahler structure; it is always quasi-regular). Hence except these two special cases, the holomorphic condition  in Theorem \ref{T-main} can be dropped. 

In this note we shall prove Theorem \ref{T-main} when the Sasaki metric $(M, g)$ is \emph{irregular}. 
Given a Sasaki metric $(M, g)$, its \emph{Reeb vector field} $\xi$ is a holomorphic Killing vector field of $(X, J, \bar g)$, where $\bar g$ is the Kahler cone metric. We fix a maximal torus $\T^k\subset Aut_0(X, J)$ such that its Lie algebra $\t$ contains $\xi$; we can assume that the dimension $k$ of $\T^k$ is at least two without loss of generality (this is the case when $\xi$ is irregular for example). Let $K$ be a maximal compact subgroup of $\text{Aut}(X, J)$ containing $\T^k$ and we denote its Lie algebra as $\h=Lie(K)$. The starting point is that the Reeb vector field is in the center of $\h$, as in quasi-regular case  \cite{MSY}.

\begin{prop}\label{P-1}The Reeb vector field $\xi$ of Sasaki-Einstein metric $(M, g)$ is in the center of $\h=Lie(K)$. 
\end{prop}

\begin{proof} Let $\z$ be the center of of $\h$. And we can then write $\t=\z\oplus \t^{'}$. The Reeb vector fields form a convex subset of $\t$, called \emph{Reeb cone} and denoted by $\cR$. As in \cite{HS}, we shall mainly interested in the \emph{normalized Reeb vector fields} which lie in a hyperplane $\cH$ in $\t$ and we denote it as $\cR^{'}=\cR\cap \cH$.  In \cite{MSY} (see \cite{FOW} for expository), it was proved that the volume functional $V: \cR^{'}\rightarrow \R$ of a Sasaki structure depending only on the Reeb vector fields, and it is a convex functional in $\cR^{'}$; moreover the Reeb vector field $\xi$ of a Sasaki-Einstein metric has to be the (unique) critical point of the volume functional. Actually it was proved further that the volume functional is actually proper in $\cR^{'}$ and hence such a minimizer always exists \cite{HS}. Clearly we can restrict our discussion on $\z$ and there is a unique minimizer, denoted as $\xi_{*}$ of the volume functional when it is restricted to the normalized Reeb cone contained in $\z$. It remains to show that $\xi=\xi_{*}$. When $\xi_*$ is quasi-regular, this is proved in \cite{MSY}. Hence we assume $\xi_*$ is irregular and hence $\text{dim}\;\z\geq 2$. We can choose a sequence normalized Reeb vector fields $\{\xi_n\}$ in $\z$ such that $\xi_n\rightarrow \xi$ by a result of Rukimbira;  moreover each $\xi_n$ can be taken as quasi-regular (see \cite{R} or Theorem 7.1.10 \cite{BG}). Now for any $\zeta\in \t^{'}$, we suppose $\zeta$ satisfies the normalized condition such that for any  normalized Reeb vector field $\tilde \xi$, $\tilde \xi+t\zeta$ is still a normalized Reeb vector field for (small) real number $t$. We then consider the volume functional $v(t)=V(\xi_n+t\zeta)$. We claim that $V(\xi_n)\leq V(\xi_n+t\zeta)$ for small $t$. Clearly $v(t)$ is a convex function of $t$ and we only need to show that $v^{'}(0)=0$. 
Since $\xi_n$ is quasi-regular and we can consider the quotient orbifold $Z=M/\cF_{\xi_n}$. Then $\t^{'}$ descends to a Lie subalgebra of $\text{aut}_\R(Z)$. Recall now the variation of the volume functional $dV$ coincides with the Futaki invariant (up to a multiplication of a constant). Now recall that the Futaki invariant $F_\C: \text{aut}(Z)\rightarrow \C$ is only nontrivial on the center of $ \text{aut}(Z)$ and in particular  it vanishes on the complexification of $\t^{'}$. Hence it follows that $dV_{\xi_n}(\zeta)=v^{'}(0)=0$ and the claim $V(\xi_n)\leq V(\xi_n+t\zeta)$ is proved. By  the smoothness of volume functional on Reeb vector fields, we know that
$V(\xi_*)\leq V(\xi_*+t\zeta)$ for any normalized $\zeta\in t^{'}$ and small $t$. It follows that $dV_{\xi_*}(\zeta)=0$ for any $\zeta\in \t^{'}$. It follows that $\xi_*$ is also a critical point of $V$ in $\cR^{'}$ (hence minimizer of $V$). By the uniqueness of minimizer in $\cR^{'}$, $\xi_*=\xi$. 
\end{proof}

Now we suppose $\xi\in \z$ and $\text{dim}(\z)\geq 2$. Let $G$ be the identity component of the isometric group of $(M, g)$ with Lie algebra $\g$;  clearly $\xi$ is also in the center of $\g$.  Now we can choose a sequence of normalized Reeb vector fields $\xi_n$ which are quasi-regular and lie in $\z$ and the center of $\g$. When $n$ is sufficiently large, then we have the following,
\begin{prop}\label{P-2}For $\xi_n$, there exists a Sasaki-Ricci soliton $g_n$ such that its underlying Kahler cone is $(X, J)$ and its identity component of the isometric group is still $G$. 
\end{prop}

\begin{proof}This is really just the local deformation of Sasaki-Ricci solitons with Kahler cone fixed while with Reeb vector fields varied. The existence of such Sasaki-Ricci solitons follows from an argument of implicit function theory (in an $G$-invariant way). The argument of Theorem 4.1 (\cite{HS}) proves such a local deformation theory in an $\T$-invariant way; since $\xi$ and $\xi_n$ are all in the center of $\g$, the same argument of Theorem 4.1 still applies with the maximal torus replaced by $G$. In particular, the isometry group of $(M, g_n)$ contains $G$. Now by a general theorem of Grove-Kratcher-Ruh \cite{GrKaRu}, we know that when $n$ large enough, there is an inclusion, up to conjugation,  of  isometry group of $(M, g_n)$ into the isometry group $G$ of $(M, g)$ (see Lemma 8.2 \cite{HS} for example). It follows that the isometry group of $(M, g_n)$ also has identity component $G$, up to conjugation. 
\end{proof}

Hence we only need to prove that the identity component of isometry group $G$ of $(M, g_n, \xi_n)$ is the identity component of a maximal compact subgroup of $\text{Aut}(X, J)$, for sufficiently large $n$. This is a Calabi type theorem \cite{Calabi85} and it  proved by Tian-Zhu \cite{Tian-Zhu2} for Kahler-Ricci solitons on Fano manifolds. 
\begin{thm}[Tian-Zhu]Suppose $(M, g, J)$ is Kahler-Ricci soliton on a Fano manifold $(M, J)$. Then the identity component of the isometry group of $(M, g)$ is  a maximal compact group of the identity component of $Aut(M, J)$.
\end{thm}

By a direct adaption of Tian-Zhu's argument,  we have
\begin{prop}\label{P-3}For quasi-regular Sasaki-Ricci solitons $(M, g_n, \xi_n)$, the identity component of its isometry group is the identity component of a maximal compact subgroup of $\text{Aut}(X, J)$.  
\end{prop}

\begin{proof}Let $K$ be a maximal group in $\text{Aut}(X, J)$ such that $\xi_n$ is in its Lie algebra $\h$ and let $K_0$ be its identity component. Then by Proposition \ref{P-1} $\xi_n$ is in $\z$, the center of $\h$. Since $\xi_n$ is quasi-regular, it generates a $U(1)$ action of $(X, J)$ contained in $K_0$. Let $Z=M/\cF_{\xi_n}$ be the quotient orbifold and let the corresponding Kahler-Ricci soliton be $h$. The compact group $K_0$, modulo $U(1)$ generated by $\xi_n$, then descends to a compact subgroup of the complex automorphism group $Aut_0(Z).$ By Tian-Zhu's theorem and its proof applied to $(Z, h)$, we know that $K_0$ acts isometrically on $(Z, h)$. It then follows that $K_0$ acts isometrically on $(M, g_n, \xi_n)$. Hence $K_0$ coincides with $G$, the identity component of isometry group of $(M, g_n, \xi_n)$. 
\end{proof}

Theorem \ref{T-main} is then a corollary of Proposition \ref{P-2} and Proposition \ref{P-3}. 

Matsushima's theorem is on Lie algebra level and does not apply directly to a finite discrete subgroup which is not contained in the identity component. 
Bando-Mabuchi \cite{BM} proved that a Kahler-Einstein metric on a Fano manifold is unique modulo automorphisms; in particular, Kahler-Einstein metric must be invariant under a discrete subgroup $\Gamma$ which is not in the identity component.  The following short argument uses the same idea as in \cite{BM}, but relies on the convexity of Ding's $\cF$-functional, established by Berndtsson \cite{Bern11}; such an argument can also be applied directly to a Kahler-Ricci soliton.

\begin{prop}\label{P-4}Let $(M, g)$ be a Kahler-Einstein metric (or a Kahler-Ricci soliton) on a Fano manifold $(M, J)$. Suppose $\Gamma$ is a discrete subgroup in $\text{Aut}(M, J)$ such that $\Gamma\cap Aut_0(M, J)=id$. Then $g$ is $\Gamma$-invariant.
\end{prop}

\begin{proof}
We assume $(M, g)$ is Kahler-Einstein for simplicity. The argument for Kahler-Ricci soliton is almost identical. 
Suppose $\l \in \Gamma$ and consider $\l^{*}g$, which is a Kahler-Einstein metric on $(M, J)$. Note that $\Gamma\subset Aut(M, J)$ and the Kahler class of $g$ and $\l^{*}g$ are both in $c_1(M, J)$, under appropriate normalization. Suppose $g\neq \l^{*}g$. 
Recall that in the space of Kahler potentials $\cH$, there exists a unique geodesic $\gamma(t)$ connecting $g, \l^{*}g$ by a fundamental result of Chen. Recall that a Kahler-Einstein metric in $c_1(M, J)$ is minimum of Ding's $\cF$-functional, which is convex along geodesics in $\cH$.   It follows that $\cF$-functional is linear (constant) along $\gamma(t)$. By Berndtsson's theorem \cite{Bern11}, $\gamma(t)$ is generated by a holomorphic vector field $\zeta$.  
In particular, there exists a one-parameter subgroup $\sigma_0$ generated by $\zeta$ such that $\sigma_0=id, \sigma_1=\l$. 
This contradicts that $\Gamma\cap Aut_0(M, J)=id$. Similar argument applies to a Kahler-Ricci soliton $(M, g)$ with $\cF$-functional replaced by modified $\cF$-functional, introduced by Tian-Zhu \cite{Tian-Zhu}. 
\end{proof}

One may wonder whether the above Bando-Mabuchi's result for Kahler-Einstein metrics on Fano manifolds holds or not for a Sasaki-Einstein metric. We believe this might not be the case in Sasaki setting due to the possible complexity of $\text{Aut}(X, J)$. The main point is that in Kahler setting, under the action of automorphism group (or discrete subgroup), the first Chern class (hence the Kahler class of Kahler-Einstein metric, modulo scaling) is invariant. In Sasaki setting, the Reeb vector field is also unique given a fixed Reeb cone; but we are not sure that such a Reeb cone is unique or not even within the Lie algebra $\t$ of a fixed (maximal) torus $\T\subset \text{Aut}(X, J)$ (see Remark 2.9 in \cite{HS}). We ask the following problem,

\begin{q}Let $(M, g)$ be a Sasaki-Einstein metric with a Reeb vector field $\xi$. Let $K$ be a maximal compact subgroup of $Aut(X, J)$ such that $\h$, the Lie algebra of $K$ contains $\xi$ in its center. Let $\Gamma$ be a discrete subgroup of $K$ such that $\Gamma \cap K_0=id$. Prove or disprove that $(M, g)$ is $\Gamma$-invariant. 
\end{q} 

We are not sure that whether the one-parameter group generated by $\xi$ is in the center of $K$ or not (we know $\xi$ is in the center of $\h$, but the proof does not carry to a finite discrete subgroup of $K$).  For any $\l\in \Gamma$,  it induces an adjoint action $Adj_\l: \h\rightarrow \h$. 
If Reeb cone contained in $\h$ is unique, then by the uniqueness of Reeb vector field of a Sasaki-Einstein metric,
 $Adj_\l (\xi)=\xi$.  It then follows that the one-parameter group generated by $\xi$ is in the center of $K$.
One can proceed to argue that $(M, g)$ is $\l$-invariant as in Fano case. However, it could happen that $\h$ contains finite many (disconnected) Reeb cones corresponding exactly to a finite group $\Gamma$, and the adjoint action $Adj_\l: \h \rightarrow \h$ permutes these Reeb cones. It would be an interesting question to understand whether this phenomenon can actually happen or not. \\

{\bf Acknowledgement:} We thank Sun Song for valuable discussions on Sasaki geometry. The author is partially supported by an NSF grant.

\end{document}